\documentclass[11pt]{amsart}

\usepackage{amsmath, amssymb, amsfonts,enumerate}
\usepackage[all]{xy}
\usepackage{amscd}
\usepackage{comment}
\usepackage{mathtools}
\usepackage{tikz} 
\usepackage{tikz-cd} 
\tikzcdset{diagrams={nodes={inner sep=7.5pt}}}

\usepackage{url}
\usepackage{hyperref}

\newtheorem{theorem}{Theorem}[section]
\newtheorem{lemma}[theorem]{Lemma}
\newtheorem{corollary}[theorem]{Corollary}

\newtheorem{theoremletter}{Theorem}

\newtheorem{corollaryletter}{Corollary}

 \theoremstyle{definition}
 \newtheorem{definition}[theorem]{Definition}

 \newtheorem{example}[theorem]{Example}

  \newtheorem*{example*}{Example}

\numberwithin{equation}{section}
 
\newcommand {\N}{\mathbb{N}}

\newcommand{\HH}{\mathcal{H}}

\newcommand\numberthis{\addtocounter{equation}{1}\tag{\theequation}}

\DeclareMathOperator{\Ker}{Ker}

\DeclareMathOperator{\Id}{Id}

\begin{document}
\title[Images of morphisms of symbolic varieties]{On images of subshifts under injective morphisms of symbolic varieties}   
\author[Xuan Kien Phung]{Xuan Kien Phung}
\email{phungxuankien1@gmail.com}
 
\begin{abstract}
We show that the image of a subshift $X$ under various  injective morphisms of symbolic algebraic varieties over monoid universes with algebraic variety alphabets is a subshift of finite type, resp. a sofic subshift, if and only if so is $X$. Similarly, let $G$ be a countable monoid and let $A$, $B$ be Artinian modules over a ring. We prove that for every closed subshift submodule $\Sigma \subset A^G$ and every injective $G$-equivariant uniformly continuous module homomorphism $\tau \colon \Sigma \to B^G$, a subshift $\Delta \subset \Sigma$ is of finite type, resp. sofic, if and only if so is the image $\tau(\Delta)$.  
Generalizations for admissible group cellular automata over admissible Artinian group structure alphabets are also obtained.  
\end{abstract}
\date{\today}
\maketitle
  
\setcounter{tocdepth}{1}
 
\section{Introduction} 
In order to state the results, we recall basic notions of symbolic dynamics. Fix a monoid $G$, called the \emph{universe}, and two sets $A, B$, the \emph{alphabets}.  
The \emph{Bernoulli right shift action} of $G$ on $A^G$, resp. on $B^G$, is defined by 
$(g, x) \mapsto g\star x$ where $(g \star x)(h) \coloneqq x(hg)$ for all   $g \in G$ and $x \in A^G$, resp.  $x \in B^G$.    
We say that a subset $\Sigma \subset A^G$ is a \emph{subshift} of $A^G$ if it is  \emph{$G$-invariant}. The sets  $A^G$, $B^G$ are equipped with the prodiscrete topology.  
\par 
Associated with a given finite subset $D \subset G$ called a
\emph{defining window}, and a subset $P \subset A^D$, we have a \emph{subshift of finite type} $\Sigma(A^G; D,P)$ of $A^G$ which is closed with respect to the prodiscrete topology and is defined by: 
\begin{align}
\label{e:sft} 
\Sigma(A^G; D,P)  \coloneqq  \{ x \in A^G \colon (g\star x)\vert_{D} \in  P \text{ for all } g \in G\}. 
\end{align} 
\par 
For sets $E \subset F$ and $\Delta \subset A^F$, we denote  the \emph{restriction} of $\Delta$ to $E$ by $\Delta_E = \{x\vert_E\colon x \in \Delta\}$. 
Let $\Sigma \subset A^G$ be a subshift. Following the work of von Neumann \cite{neumann}, a map $\tau \colon \Sigma \to B^G$ is \emph{a cellular automaton} if it  admits a finite \emph{memory set} $M \subset G$
and a \emph{local defining map} $\mu \colon \Sigma_M \to B$ such that 
\begin{equation*} 
\label{e:local-property}
(\tau(c))(g) = \mu((g\star  c )\vert_M)  \quad  \text{for all } c \in \Sigma \text{ and } g \in G.
\end{equation*} 
\par  
Equivalently, a map $\tau \colon \Sigma \to B^G$ is a cellular automaton if and only if 
it is $G$-equivariant and uniformly continuous with respect to the prodiscrete uniform structure 
(cf.~\cite[Theorem~4.6]{csc-monoid-surj}).  
\par 
Now let $G$ be a monoid and let $X, Y$ be algebraic varieties over an algebraically closed  field $k$, i.e., reduced $k$-schemes of finite type \cite{grothendieck-ega-1-1}. We denote by $A=X(k)$ and $B=Y(k)$ the sets of $k$-points of $X$ and $Y$. Note that we can identify $X$, $Y$ with $A$, $B$ respectively. 
\par 
We say that a subshift $\Sigma \subset A^G$ is an \emph{algebraic subshift} if for every finite subset $E \subset G$, the restriction 
$\Sigma_E$ is a subvariety of $A^E$.
For an algebraic subshift $\Sigma \subset A^G$, a map $\tau \colon \Sigma \to B^G$ is called an \emph{algebraic cellular automaton} if there exists a finite memory  subset $M \subset G$ and a local defining map which is a morphism of algebraic varieties $\mu \colon \Sigma_M \to B$ such that
\[
    \tau(x)(g) = \mu((g \star x)\vert_M), \quad \text{for all } x \in \Sigma \text{ and } g \in G. 
\] 
\par 
Images of subshifts of finite type under cellular automata are called \emph{sofic subshifts} \cite{weiss-sofic-shift}. When a sofic subshift is not of finite type, we say that it is \emph{stricly sofic}. See  \cite{lind-marcus} and  \cite{weiss-sofic-shift} for various examples of stricly sofic subshifts. 
\par 
The first main result of the paper is the following theorem which asserts that the image $\tau(\Delta)$ of a subshift $\Delta$ under an injective algebraic cellular automaton $\tau$ is a subshift of finite type if and only if so is $\Delta$. In particular, the image of a subshift of finite type under an injective algebraic cellular automaton 
cannot be strictly sofic. 
\par 
More precisely, we will show in Section~\ref{s:main-morphism-symbolic-var} that: 
 
\begin{theoremletter}
\label{t:sft-image-intro-main-alg}
Let $G$ be a countable monoid and let $X, Y$ be algebraic varieties  over an uncountable algebraically closed field $k$. Let $\Sigma \subset X(k)^G$ be a closed algebraic subshift. Suppose that $\tau \colon \Sigma \to Y(k)^G$ is an injective algebraic cellular automaton. Then a subshift $\Delta \subset X(k)^G$ contained in $\Sigma$ is of finite type, resp. sofic, if and only if so is the image subshift  $\tau(\Delta) \subset Y(k)^G$.    
\end{theoremletter}
\par 
Since the full shift is a subshift of finite type, we obtain as an immediate application of Theorem~\ref{t:sft-image-intro-main-alg} the following result. 

\begin{corollaryletter}
\label{c:sft-image-intro-corollary-1} 
Let $G$ be a countable monoid and let $X,Y$ be algebraic varieties over an uncountable algebraically closed field $k$. Let $A=X(k)$ and $B=Y(k)$. Suppose that $\tau \colon A^G \to B^G$ is an injective algebraic cellular automaton. Then  $\tau(A^G)$ is a subshift of finite type of $B^G$.    \qed
\end{corollaryletter}
\par 
Note that whenever $X=Y$ are finite and when $G$ is moreover a group,  Corollary~\ref{c:sft-image-intro-corollary-1} implies    \cite[Theorem~2.3]{doucha}. 
\par 
The second goal of the paper is to establish a similar result to  Theorem~\ref{t:sft-image-intro-main-alg} for injective admissible group cellular automata over countable monoid universes and admissible Artinian group structure alphabets (see Section~\ref{s:admissible-subshift}). We formulate below a particular application of the general Theorem~\ref{t:main-2-intro-admissible-general}.

\begin{theoremletter}
\label{t:main-corollary-admissible-intro}
Let $G$ be a countable monoid and let $A$, $B$ be Artinian modules over a ring. Let $\Sigma \subset A^G$ be a closed  subshift  submodule, e.g., $\Sigma= A^G$.  Suppose that $\tau \colon \Sigma \to B^G$ is an injective cellular automaton which is also a module homomorphism. Then for every subshift $\Delta \subset A^G$ contained in $\Sigma$, the subshift  $\tau(\Delta) \subset B^G$ is of finite type, resp.  sofic, if and only if so is $\Delta$.  
\end{theoremletter}
\par 
The present paper belongs to the rich literature on the study of injective morphisms of symbolic varieties and cellular automata. The topic admits a long history and substantial developments which date back to the works of Moore
\cite{moore} and Myhill \cite{myhill} on the well-known Garden of Eden theorem. 
Over group universes, various  injective endomorphisms of symbolic varieties are  bijective as motivated by the  Gottschalk surjunctivity conjecture \cite{gottschalk} and the seminal paper of Gromov \cite{gromov-esav} (see also \cite{csc-sofic-linear}, \cite{cscp-alg-ca}, \cite{cscp-alg-goe},  \cite{phung-2020} \cite{phung-post-surjective}, \cite{phung-LEF}). Notable applications of  surjunctivity property include the well-known Kaplansky's stable finiteness conjecture \cite{kap} on group rings (see \cite{ara}, \cite{elek},  \cite{phung-2020}, \cite{cscp-jpaa},  \cite{phung-geometric}, \cite{phung-weakly}) and a dynamical characterization of amenable groups \cite{bartholdi-kielak}.  
\par 
Over finite alphabets, bijective cellular automata over group universes are  automorphisms (see e.g.  \cite[Theorem~1.3]{cscp-alg-ca} for more general alphabets).  
However, when the universe is merely a monoid, we know many examples of injective non-surjective cellular automata (see \cite{csc-monoid-surj} whenever the monoid universe contains a bicyclic submonoid) which provide us with interesting strict embeddings of subshifts. Hence,  we find that it is natural to investigate the relations between subshifts and their images under such embeddings,  which constitutes the main motivation of the paper. 
\par 
The paper is organized as follows. Section~\ref{s:preliminary} provides some basic lemmata and results on the induced local maps and subshifts of finite type. Section~\ref{s:criterion-sft} presents a useful criterion (Theorem~\ref{t:criterion-finite-type}) for a subshift to be of finite type that we will apply frequently in the proof of the main results. 
In  Section~\ref{s:left-inverse-algebraic}, we establish the left reversibility of injective morphisms of symbolic varieties. 
Section~\ref{s:main-morphism-symbolic-var} contains the proof of Theorem~\ref{t:sft-image-intro-main-alg}. 
Basic definitions and properties of admissible Artinian group structures and admissible group cellular automata are collected in Section~\ref{s:admissible-subshift}. Then we formulate and prove a left reversibility result (Theorem~\ref{l:techno-left-inverse}) for injective admissible group cellular automata in Section~\ref{s:techno-left-inverse}. We establish in  Section~\ref{s:main-2-intro-admissible-general} the second main result of the paper  Theorem~\ref{t:main-2-intro-admissible-general} from which we deduce a proof of   Theorem~\ref{t:main-corollary-admissible-intro} given in Section~\ref{s:proof-thm-b}. Finally, in Section~\ref{s:app-sym-group-var} we give another application of Theorem~\ref{t:main-2-intro-admissible-general} to obtain an improvement of  Theorem~\ref{t:sft-image-intro-main-alg} in the case of injective morphisms of symbolic group varieties (Theorem~\ref{t:sym-grp-var-inj}). 
 \section{Preliminaries}
 \label{s:preliminary}

The set of non-negative integers is denoted 
by $\N$. For subsets $E, F$ of a monoid $G$, we denote their product  by 
\[
EF \coloneqq \{ xy\colon x \in E, y \in F \} \subset G.
\]
 \subsection{Subshifts of finite type} 
We have the following elementary observation which allows us to perform the base change of defining windows for subshifts of finite type. 
 \begin{lemma}
\label{l:window-change-sft}
Let $G$ be a monoid and let $A$ be a set. Let $\Sigma= \Sigma(A^G; D,P)$ for some finite subset $D \subset G$ and some subset $P \subset A^D$. Then for every subset $E \subset G$ such that $D \subset E$, we have $\Sigma= \Sigma(A^G; E, \Sigma_E)$. 
\end{lemma}
 
\begin{proof}
See \cite[Lemma~2.1]{phung-israel}, \cite[Lemma~5.1]{cscp-invariant-ca-alg} for the case when $G$ is a group. The proof is similar when $G$ is a monoid. 
\end{proof} 
\par 
The following  remark will be useful for the proof of our main results introduced in the Introduction. 

\begin{lemma}
\label{l:restriction-subshift-equality}
Let $G$ be a monoid and let $A$ be a set. Suppose that $\Delta \subset A^G$ is a subshift. Let $F  \subset G$ and $\Lambda= \Sigma(A^G; F, \Delta_F)$. Then for every subset $E \subset F$, we have $\Lambda_E= \Delta_E$.
\end{lemma}

\begin{proof}
Observe first that $\Delta \subset \Lambda$ since every configuration $x \in \Delta$ satisfies trivially $(g\star x)\vert_F \subset \Delta_F$ for all $g \in G$ as $g \star x \in \Delta$. It follows that $\Delta_E \subset \Lambda_E$. 
On the other hand, since $ E \subset F$ by hypothesis, we find that
\begin{align*}
\Lambda_E  & = (\Lambda_{F})_E & (\text{since } E \subset F)  \\ & \subset (\Delta_{F})_E & (\text{as }\Lambda= \Sigma(A^G; F, \Delta_{F}))  
\\ & = \Delta_E
& (\text{since } E \subset F) 
\\ & \subset \Lambda_E 
& (\text{since } \Delta \subset \Lambda).   
\end{align*}
\par 
Consequently, we have $\Lambda_E= \Delta_E$ and the proof is complete. 
\end{proof}

\subsection{Induced local maps} 
\label{s:induced-map}
For the notation, let $G$ be a monoid and let $A, B$ be sets. Let $\Sigma \subset A^G$ be a subshift and let 
$\tau \colon \Sigma \to B^G$ be a cellular automaton. Fix a  memory set $M$ and the corresponding  local defining map $\mu \colon \Sigma_M \to B$. For every finite subset $E \subset G$, we denote by $\tau_E^+ \colon \Sigma_{ME} \to B^E$ the induced local map of $\tau$ by setting 
$\tau_E^+(x)(g) = \mu ((g \star y)\vert_M)$ for every $x \in \Sigma_{ME}$, $g \in E$, and  $y \in \Sigma$ such that $y \vert_{ME}=x$.  Equivalently, we can define  
\[ 
\tau_E^+(x\vert_{ME})= \tau(x)\vert_E, \quad \text{ for all }x \in \Sigma. 
\] 
\par 
We have the following auxiliary lemma for the induced maps of algebraic cellular automata:

  \begin{lemma}
  \label{l:induced-alg-map}
Let $G$ be a monoid and let $X, Y$ be algebraic varieties over an algebraically closed field $k$. Let $A=X(k)$, $B= Y(k)$, and let $\Sigma \subset A^G$ be an algebraic subshift. Fix a memory set $M \subset G$ of an algebraic cellular  automaton $\tau\colon \Sigma \to B^G$.  
Then for every finite subset $E \subset G$, the induced map  
$\tau_{E}^{+} \colon \Sigma_{M E} \to B^E$  
is a morphism of $k$-algebraic varieties.  
\end{lemma} 

\begin{proof}
The lemma follows directly from the universal property of fibered products. The morphism $\tau_E^+$ is determined by the component morphisms $T_g\colon \Sigma_{ME} \to A^{\{g\}}$, $g \in E$, given by  $T_g(x)  = \mu ((g\star x)\vert_M)$ for all $x \in \Sigma_{ME}$. It suffices to note that $T_g$ is the composition of the morphism $\Sigma_{Mg} \to B^{\{g\}}$ induced by the morphism $\mu$ and the canonical projection $\Sigma_{ME} \to \Sigma_{Mg}$ which is clearly algebraic.  
\end{proof}

\section{A criterion for  subshifts to be of finite type} 
\label{s:criterion-sft}
In this section, we formulate a general technical criterion (Theorem~\ref{t:criterion-finite-type}) for subshifts to be of finite type that will be useful for the proof of the main results of the paper. 
\par 
Let us first introduce the context and notation. Given a monoid $G$ and two sets $A$, $B$.  Let $\Sigma$ and $\Gamma$ be respectively  subshifts of $A^G$ and $B^G$. Suppose that $\tau \colon \Sigma \to B^G$ and $\sigma \colon \Gamma \to A^G$ are cellular automata with a common memory set $M \subset G$ such that $1_G \in M$. 

\begin{theorem}
\label{t:criterion-finite-type}
With the above notation, suppose that 
$\Gamma= \tau(\Sigma)$ and $\sigma \circ \tau$ is the identity map on $\Sigma$. Assume in addition that   $\Sigma= \Sigma(A^G; M, \Sigma_M)$. Then one has $\Gamma = \Sigma(B^G; M^2, \Gamma_{M^2})$. Thus, $\Gamma \subset B^G$ is a subshift of finite type. 
\end{theorem}

\begin{proof}
Let us denote $\Lambda= \Sigma(B^G; M^2, \Gamma_{M^2})$. Then $\Lambda$ is a subshift of finite type of $B^G$ and it is clear that $\Gamma \subset \Lambda$. 
\par 
Let $\mu_M \colon \Sigma_M \to B$ and $\eta_M \colon \Gamma_M \to A$ be respectively the local defining maps of $\tau$ and $\sigma$ associated with the memory set $M$. 
Note that since $1_G \in M$, we have $M \subset M^2$. It follows that $\Lambda_M= \Gamma_M$ by Lemma~\ref{l:restriction-subshift-equality}. 
\par 
Consequently,  we obtain a well-defined cellular automaton $\pi \colon \Lambda \to A^G$  which admits $\eta_M \colon \Lambda_M \to A$ as the local defining map associated with the finite memory set $M \subset G$. Observe also that $\pi\vert_\Gamma = \sigma$ since the cellular automata $\pi$ and $\sigma$ have the same local defining map and $\Gamma \subset \Lambda$. 
\par 
We claim that  $\pi (\Lambda) \subset \Sigma$. 
Indeed, let $y \in \Lambda$ and let $g \in G$. Since we have $\Lambda = \Sigma(B^G; M^2, \Gamma_{M^2})$ by definition, 
$(g \star y)\vert_{M^2} \in \Lambda_{M^2} \subset \Gamma_{M^2}$. Therefore,  $(g \star y)\vert_{M^2} = x\vert_{M^2}$ for some  configuration $x \in \Gamma$.
\par 
Since $\Gamma = \tau(\Sigma)$ by hypotheses, we can choose $z \in \Sigma$ such that $\tau(z)=x$. We note that 
\[
\pi(\tau(z))=\sigma(\tau(z))=z
\]
since $z \in \Sigma$ and since $\pi \circ \tau= \sigma \circ \tau$ acts as the identity map on $\Sigma$. We can thus compute:  
\begin{align}
\label{e:g-star-pi-y-1}
   (g \star  \pi(y)) \vert_{M} & = \pi(g \star y)\vert_{M}  
    = \pi_M^+((g\star y)\vert_{M^2}) \nonumber \\ 
   & = \pi_M^+(x\vert_{M^2})
    = \pi(x)\vert_M \nonumber \\ & = \pi(\tau(z))\vert_M= z\vert_M. \numberthis
\end{align}
\par 
Consequently, $(g \star \pi(y)) \vert_{M} = z\vert_M \in \Sigma_M$ for all $g \in G$. Thus, it  follows from the hypothesis 
 $\Sigma= \Sigma(A^G; M, \Sigma_M)$ 
that $\pi(y) \in \Sigma$ for all $y \in \Lambda$. Therefore, $\pi(\Lambda) \subset \Sigma$ and the claim is proved. 
\par 
Now let $y \in \Lambda$ and let $g \in G$. Since $(g \star y)\vert_{M^2} \in Y_{M^2} \subset \Gamma_{M^2}$, we can find as above $x \in \Gamma$ and $z\in \Sigma$ such that $\tau(z)=x$ and  $(g \star y)\vert_{M^2} = x\vert_{M^2}$. In particular, $(g \star y)(1_G)=x(1_G)$ as $1_G \in M^2$. 
\par 
We have seen in \eqref{e:g-star-pi-y-1} that $(g \star  \pi(y)) \vert_{M} = z\vert_M$. As $\pi(y) \in \Sigma$, it makes sense to  write and consider $\tau(\pi(y))$ that  we can compute as follows: 

\begin{align*}
\label{e:main-ad-grp-proof-2}
    \tau(\pi(y))(g) 
    & = \mu_M ((g \star \pi(y))\vert_M)  \\ 
    &  = \mu_M(z\vert_M)  
     = \tau(z)(1_G)  \\
      & = x(1_G)   
       = (g\star y)(1_G)\\
       & = y(g).
\end{align*}
\par 
Hence, we have $y= \tau(\sigma(y)) $ for all $y \in \Lambda$. It follows that $y \in \tau(\Sigma) = \Gamma$ and therefore $\Lambda \subset \Gamma$. On the other hand, $\Gamma \subset \Sigma(A^G; M^2, \Gamma_{M^2})= \Lambda$ so  we can conclude that 
\[
\Gamma = \Lambda= \Sigma(B^G; M^2, X_{M^2}). 
\]
\par 
In particular, $\Gamma$ 
is a subshift of finite type of $B^G$. The proof is thus complete. 
\end{proof}

\section{Left inverses of injective morphisms of symbolic varieties} 
\label{s:left-inverse-algebraic}

In this section, we shall establish the following left reversibility result for injective algebraic cellular automata. 

\begin{theorem}
\label{t:left-reverse-ca-alg} 
Let $G$ be a countable monoid.   
Let $X, Y$ be algebraic varieties  over an uncountable algebraically closed field $k$ and  let $A=X(k)$, $B=Y(k)$.  Let $\Sigma\subset A^G$ be a closed algebraic subshift and let $\Gamma= \tau(\Sigma)$.  Suppose that  
 $\tau \colon \Sigma \to B^G$ is an injective algebraic cellular automaton. 
Then there exists a finite subset  $N\subset G$ such that 
for every finite subset $E \subset G$ containing $N$, 
there exists a map   
$\eta_E \colon  \Gamma_E \to A$ with 
$\eta_E(\tau(x)\vert_E)=x(1_G)$  for all $x\in \Sigma$. 
\end{theorem}

We begin with the following technical lemma from which Theorem~\ref{t:left-reverse-ca-alg} will follow without difficulty. 

\begin{lemma}
\label{l:symbol-var-left-inverse}
Let the notation and hypotheses be as in Theorem~\ref{t:left-reverse-ca-alg}. 
Then there exists $N\subset G$ finite and such that 
 $\tau^{-1}(x)(1_G)\in A$ depends uniquely on the restriction $x \vert_N$ for every configuration  $x\in \Gamma$.   
\end{lemma}

\begin{proof} 
Since  $\tau \colon \Sigma \to B^G$ is an algebraic cellular automaton, it admits a local defining map $\mu \colon \Sigma_M \to B$ associated with a memory set $M \subset G$ such that $1_G \in M$ and such that $\mu$ is a $k$-morphism of algebraic varieties. 
\par
As $G$ is countable, there exists an increasing sequence of finite subsets $(E_n)_{n \in \N}$ of $G$ such that $G= \cup_{n \in \N} E_n$ and $M \subset E_0$. 
\par 
For every $n \in \N$, 
we have a 
$k$-morphism $\tau_{E_n}^+ \colon \Sigma_{ME_n} \to B^{E_n}$ of algebraic varieties defined in Section \ref{s:induced-map}. Then $\tau_{E_n}^+ $ induces a  $k$-morphism of algebraic varieties 
\begin{equation*}
\Phi_n \coloneqq \tau_{E_n}^+ \times \tau_{E_n}^+ \colon  \Sigma_{ME_n}\times_k  \Sigma_{ ME_n}\to  B^{E_n}\times_k B^{E_n }.
\end{equation*} 
\par 
For every finite subset $E \subset G$, we denote respectively by $\Delta_E$ and $\Gamma_E$ the diagonal of $\Sigma_E \times_k \Sigma_E$ and $B^E \times_k B^E$. 
Consider the canonical projections  $\pi_{n} \colon \Sigma_{ME_n }\times_k  \Sigma_{ME_n}\to \Sigma_{\{1_G\}} \times_k  \Sigma_{\{1_G\}}$. Since $\pi_n$ and $\Phi_n$ are clearly algebraic, we obtain a constructible subset of $\Sigma_{ME_n}\times_k  \Sigma_{ME_n}$ given by:  
\begin{equation*} 
V_n \coloneqq  \Phi_n^{-1}(\Gamma_{E_n})\setminus \pi_{n}^{-1}(\Delta_{\{1_G\}}).
\end{equation*} 
\par 
By construction, we note that the set of  closed points of $V_n$ consists of 
 the couples $(u,v)$ with $u,v \in   \Sigma_{ME_n} $ such that $\tau_{E_n}^+(u)=\tau_{E_n}^+(v)$ and $u(1_G) \neq v(1_G)$. 
\par 
For the proof, we proceed by  supposing on the contrary that there does not exist a finite subset $N$ which satisfies the conclusion of the lemma. 
\par 
Therefore, the sets $V_n$ are nonempty for all $n \in \N$ and 
 we obtain a projective  system 
$(V_n)_{n \in \N }$ of nonempty constructible subsets of the  $k$-algebraic varieties $\Sigma_{ME_n}\times_k  \Sigma_{ME_n}$ 
with transition maps 
$p_{m,n} \colon V_m \to V_n$, for $m \geq n \geq 0$, induced by the canonical projections $\Sigma_{ME_m }\times_k  \Sigma_{ME_m} \to 
\Sigma_{ME_n }\times_k  \Sigma_{ME_n}$. 
\par 
Since the $k$ is an uncountable and algebraically closed field,  \cite[Lemma~B.2]{cscp-alg-ca} (see also  \cite[Lemma~3.2]{cscp-invariant-ca-alg}) implies that $\varprojlim_n V_n \neq \varnothing$. On the other hand, since $\Sigma$ is closed in the prodiscrete topology by hypothesis, we infer from \cite[Lemma~2.5]{phung-israel} that 
$\varprojlim_n \Sigma_{E_n M} = \Sigma$ and hence: 
\begin{equation*}
 \varprojlim_n V_n \subset \varprojlim_n \Sigma_{E_n M} \times \Sigma_{E_n M} = \Sigma \times \Sigma, 
\end{equation*}
\par 
Consequently, by the construction of the sets $V_n$, we can find $x,y \in \Sigma$ such that $(x,y) \in \varprojlim_n V_n$  thus  $\tau(x)= \tau(y)$ and $x(1_G) \neq y(1_G)$. In particular, $x \neq y$ and as a result, the map $\tau$ is not injective, which is a contradiction. The proof is thus complete. 
\end{proof} 
\par

We are now in the position to give the proof of Theorem~\ref{t:left-reverse-ca-alg}. 

\begin{proof}[Proof of Theorem~\ref{t:left-reverse-ca-alg}]
Let $N \subset G$ be the finite subset given by Lemma~\ref{l:symbol-var-left-inverse}. 
Then for every finite subset $E \subset G$ with $N \subset E$, we obtain a well-defined map: 
\begin{equation*}
     \eta_E \colon \Gamma_E \to A,\quad 
     x \mapsto \tau^{-1}(y)(1_G),
\end{equation*}
where $y \in \Gamma$ is an arbitrary configuration such that $y\vert_E= x\vert_E$ and $\tau^{-1}(y)$ denotes the unique element of $\Sigma$ whose image is $y$. Note that since $\tau$ is injective, $\tau^{-1}(y)$ is well-defined. 
Consequently, for $x \in \Sigma$ and $y= \tau(x)\vert_E$, we can write $\tau^{-1}(\tau(x))=x$ and thus: 
\begin{equation*}
    \eta_E(\tau(x)\vert_E)= \eta_E(y)=  \tau^{-1}(\tau(x))(1_G)=x(1_G).
\end{equation*} 
\par 
Hence, the proof is complete.  
\end{proof}

\section{Images of injective morphisms of symbolic varieties}
\label{s:main-morphism-symbolic-var}

We shall establish in this section the following finiteness result on the images of injective algebraic cellular automata which is the first main result of the paper. 

\begin{theorem}
\label{t:sft-image} 
Let $G$ be a countable monoid and let $X, Y$ be  algebraic varieties over an uncountable algebraically closed field $k$. Let $A= X(k)$, $B=Y(k)$ and let $\Sigma \subset A^G$ be a closed algebraic subshift. Suppose that $\tau \colon \Sigma \to B^G$ is an injective algebraic cellular automaton. Then for every subshift of finite type $\Delta \subset \Sigma$ of $A^G$, the image $\tau(\Delta)\subset B^G$ is a subshift of finite type.   
\end{theorem}

\begin{proof} 
As $G$ is countable, there exists an increasing sequence of finite subsets $(E_n)_{n \in \N}$ of $G$ such that $G= \cup_{n \in \N} E_n$ and $1_G \subset E_0$. 
\par 
Let us fix an algebraic local defining map $\mu \colon \Sigma_M \to B$ of $\tau$ associated with a finite memory set $M\subset G$ such that $1_G \in M$. 
As $\Delta$ is a subshift of finite type, it admits a defining window $D\subset G$ so that 
$\Delta= \Sigma(A^G; D, \Delta_D)$ (see the definition  \eqref{e:sft}). We denote also  $\Gamma= \tau(\Sigma)$ and $X= \tau (\Delta)$. 
\par 
By Theorem~\ref{t:left-reverse-ca-alg}, we can find a finite subset $N \subset G$ such that 
for every finite subset $E \subset G$ with $N \subset E$, 
there exists a map   
$\eta_E \colon  \Gamma_E \to A$
such that for every $x\in \Sigma$, we have: 
\begin{equation}
\label{e:left-inverse-local-alg-1}
\eta_E(\tau(x)\vert_E)=x(1_G).
\end{equation} 
\par 
Up to enlarging $M$ and $N$, we can suppose without loss of generality that $D \subset M=N$. Hence, by Lemma~\ref{l:window-change-sft},  we can write: 
\begin{equation}
\label{e:main-ad-group-sigma-fsft-1}
    \Delta = \Sigma(A^G; M, \Delta_M).
\end{equation}
\par 
We define $Y\coloneqq \Sigma(B^G; M^2, X_{M^2})$ then it is clear that $Y \subset X$ is a subshift of finite of $B^G$. In the sequel, we will show that $Y \subset X$ and consequently $X=Y$ will be a subshift of finite type. 
\par 
Let us consider $\Lambda \coloneqq \Sigma(B^G; M^2, \Gamma_{M^2})\subset B^G$ and the cellular automaton $\sigma \colon \Lambda \to A^G$ which admits $M$ as a memory set and  $\eta_M\colon \Gamma_M \to A$ as the corresponding local defining map. Note that $\Gamma_M= \Lambda_M$ since $1_G \in M$ (cf. the proof of Theorem~\ref{t:criterion-finite-type}). 
\par 
We claim that 
$\sigma \circ \tau = \Id \colon \Sigma \to \Sigma$ is the identity map on $\Sigma$.
Indeed, 
we infer from  the $G$-equivariance of the cellular automata $\tau$ and $\sigma$ and 
from the property  \eqref{e:left-inverse-local-alg-1} 
that for all $x \in \Sigma$ and  $g \in G$, we have: 
\begin{align} 
\label{e:main-ad-grp-proof-1}
    \sigma(\tau(x))(g)
    & = \eta_M((g \star \tau(x))\vert_M) \nonumber \\
    & = \eta_M(\tau(g \star x)\vert_M)\nonumber \\
    & = (g \star x)(1_G) \nonumber \\
    &= x(g). \numberthis 
\end{align} 
\par 
Since $g\in G$ is arbitrary, $\sigma(\tau(x))=x$ and the claim is thus proved. In particular, since   $\Delta \subset \Sigma$,  the restriction  $\sigma \circ  \tau\vert_\Delta \colon \Delta \to A^G$ acts as the identity map on $\Delta$. 
\par 
Since $X=\tau(\Delta)$ and $\Delta= \Sigma(A^G; M, \Delta_M)$ by definition, we infer from Theorem~\ref{t:criterion-finite-type} applied to   $\tau\vert_\Delta \colon \Delta \to B^G$, $\sigma\vert_X\colon X \to A^G$, and $\Delta$ that 
\[
\tau(\Delta) = X = \Sigma(B^G; M^2, X_{M^2}).   
\]
\par 
Therefore, the image $\tau(\Delta)$ is a subshift of finite type of $B^G$. The proof is thus complete. 
 \end{proof}
 \par 
We prove below that the converse of Theorem~\ref{t:sft-image} also holds. Moreover, we see that not only being a subshift of finite type but also soficity are preserved under injective morphisms of symbolic algebraic varieties. Theorem~\ref{t:sft-image-intro-main-alg} in the Introduction is the consequence of the following result. 

\begin{theorem}
\label{t:converse-sft-image} 
Let $G$ be a countable monoid. Let $X, Y$ be algebraic varieties  over an uncountable algebraically closed field $k$. Suppose that $\tau \colon \Sigma \to Y(k)^G$ is an injective algebraic cellular automaton where $\Sigma \subset X(k)^G$ is a closed algebraic subshift. Then for every subshift $\Delta \subset X(k)^G$ such that $\Delta \subset \Sigma$, the following hold: 
\begin{enumerate} [\rm (i)] 
    \item $\Delta$ is a subshift of finite type if and only if  so is  $\tau(\Delta)$; 
    \item 
    $\Delta$ is a sofic subshift if and only if so is  $\tau(\Delta)$.   
\end{enumerate} 
\end{theorem}

 \begin{proof}
 By Theorem~\ref{t:sft-image}, we know that if $\Delta$ is a subshift of finite type then so is the image   $\tau(\Delta)$. This proves one of the two implications of  (i). Now suppose that $X=\tau(\Delta)$ is a subshift of finite type. We denote $A=X(k)$, $B=Y(k)$, and $\Gamma= \tau(\Sigma)$. Then there exists a finite subset  $F\subset G$ such that 
$X= \Sigma(B^G; F, X_F)$. 
Let $\mu_M \colon \Sigma_M \to A$ be the  local defining map of $\tau$ associated with a  memory set $M\subset G$ such that $1_G \in M$. 
Theorem~\ref{t:left-reverse-ca-alg} implies that  there exists a finite subset $N \subset G$ such that 
for every finite subset $E \subset G$ containing  $E$, 
we have a map   
$\eta_E \colon  \Gamma_E \to A$
such that $
\eta_E(\tau(x)\vert_E)=x(1_G)$  for every $x\in \Sigma$. 
By replacing $M$ and $N$ by $(M \cup N)F$,  we can suppose without loss of generality that $D \subset M=N$. We infer from   Lemma~\ref{l:window-change-sft} that: 
\begin{equation}
\label{e:general-symbolic-1-1}
      X= \Sigma(A^G; M, X_M). 
\end{equation}
\par 
Let us define $Z\coloneqq \Sigma(A^G; M^2, \Delta_{M^2}) \subset A^G$ then  $\Delta \subset Z$. Moreover, we infer from Lemma~\ref{l:restriction-subshift-equality} that $Z_M = \Sigma_M$ since $M \subset M^2$ as $1_G \in M$. 
\par 
Therefore, the local defining map $\mu_M$ determines a cellular automaton 
$\pi \colon Z \to A^G$ whose restriction to $\Delta$ coincides   with $\tau$, that is, $\pi\vert_\Delta= \tau$.   
\par 
Denote $\Lambda \coloneqq \Sigma(A^G; M^2, \Gamma_{M^2}) \subset A^G$ and consider the cellular automaton $\sigma \colon \Lambda \to A^G$ admitting $\eta_M\colon \Gamma_M \to A$ as a local defining map (note that $\Gamma_M=\Lambda_M$ by Lemma~\ref{l:restriction-subshift-equality} as $M \subset M^2$).  
\par 
Since 
$\sigma \circ \tau$ is the identity map on $\Sigma$ as we have seen in the proof of Theorem~\ref{t:main-2-intro-admissible}, we have $\sigma(X)=\sigma(\tau(\Delta))= \Delta$. On the other hand, since $\tau(\Delta)= X$ and $\Delta \subset \Sigma$, 
we deduce immediately that the restriction 
$\pi \circ \sigma \vert_X= \tau \circ \sigma \vert_X$ acts as the identity map on $X$.    
\par 
Hence, it follows from Theorem~\ref{t:criterion-finite-type} applied to  the subshift of finite type $X$ and the cellular automata $\sigma\vert_X \colon X \to A^G$, $\tau\vert_\Delta \colon \Delta \to A^G$ that 
\[ 
\Delta = Z= \Sigma(A^G; M^2, \Delta_{M^2})
\] 
so $\Delta$ is a subshift of finite type. The point (i) is  proved. 
\par 
For (ii), assume that $\Delta$ is sofic so $\Delta= \gamma(W)$ for some subshift of finite type $W$ and some cellular automaton $\gamma$.  Since compositions of cellular automata are also cellular automata, we find that $\tau(\Delta)= \tau(\pi(X))$ is a sofic subshift. 
\par 
Conversely, suppose that $ \tau(\Delta) \subset A^G$ is a sofic subshift. Then  $\tau(\Delta)$ is the image of a subshift of finite type $W$ under a cellular automaton $\gamma$.  
Since $\Delta \subset \Sigma$ and $\sigma\circ \tau =\Id_\Sigma$, it follows that $\Delta = \sigma (\tau(\Delta)) =
\sigma (\gamma(W))$ 
is a sofic subshift. The proof is thus complete. 
 \end{proof}

\section{Admissible group subshifts}
\label{s:admissible-subshift} 
In this section, we recall and formulate direct extensions to the case of monoid universes the   notion of admissible group subshifts introduced in \cite{phung-2020} 
as well as their basic properties (see also \cite{phung-dcds}). 
 
 \subsection{Admissible Artinian group structures}  
\begin{definition} [cf.  \cite{phung-2020}, \cite{phung-dcds}]
\label{d:general-artinian-structure}
Given a group $A$. Suppose that for every $n\geq 1$, 
$\HH_n$ is a collection of subgroups of $A^n$ with the following properties: 
\begin{enumerate} 
\item 
$ \{1_A\}, A \in \HH_1$, and $\Delta \in \HH_2$ where   $\Delta = \{(a,a) \in A^2 \colon a \in A \}$ is the diagonal subgroup of $A^2$; 
\item 
for $m \geq n \geq 1$ and for every projection $\pi \colon A^m \to A^n$ induced by any  injection 
$\{1, \dots, n\} \to \{ 1, \dots, m\}$, one has $\pi(H_m) \in \HH_n$ and $\pi^{-1}(H_n) \in \HH_m$ 
for every $H_m \in \HH_m$ and $H_n \in \HH_n$;   
\item 
for each $n \geq 1$ and $H, K\in \HH_n$, one has $H \cap K \in \HH_n$; 
\item 
for each $n \geq 1$, every descending sequence $(H_k)_{k \geq 0}$,  where  
 $H_k \in \HH_n$ for every $k \geq 0$, eventually stabilizes. 
\end{enumerate} 
\par 
Let $\mathcal{H} \coloneqq (\mathcal{H}_n)_{n \geq 1}$. We say that  $(A, \HH)$, or $A$ if the context is clear, 
is an \emph{admissible Artinian group structure}. 
For every $n \geq 1$, elements of $\HH_n$ are called \emph{admissible subgroups} of $A^n$. 
\par
Note that in our definition, we require the extra condition $\Delta \in \HH_2$ in comparison to  \cite[Definition 9.1]{phung-2020}. 
\par 
If $E$ is a finite set,  then $A^E$ admits an admissible Artinian structure induced by that of $A^{\{1, \dots, |E|\}}$ via an arbitrary bijection 
$\{1, \dots, |E|\} \to E$.    
\end{definition}

\begin{example} 
\label{ex:canonical-admissible} 
(cf.~\cite[Examples~9.5,~9.7]{phung-2020}) 
An algebraic group $V$ over an algebraically closed field, 
resp. a compact Lie group $W$, resp. an Artinian (left or right) 
module $M$ over a ring $R$, 
admits a canonical admissible Artinian structure given by all algebraic subgroups of $V^n$, 
resp. by all closed subgroups of $W^n$, resp. 
by all $R$-submodules of $M^n$, for every $n \geq 1$. 
\end{example}

\begin{definition}  [cf. \cite{phung-2020}]
\label{d:admissible-homomorphism} 
Let $(A, \HH)$ be an admissible Artinian group structure.  
Let $m, n \geq 0$ and let $X, Y$ be respectively admissible subgroups of $A^m$ and $A^n$. 
We say that a group homomorphism  $\varphi \colon X \to Y$ 
is $\HH$-\emph{admissible} (or simply \emph{admissible}) if the graph $\Gamma_\varphi \coloneqq \{ (x, \varphi(x)) \colon x \in X \}\subset X \times Y$ 
is an admissible subgroup of $A^{m+n}$.  
\end{definition}
\par 
In Definition~\ref{d:admissible-homomorphism}, suppose that $\varphi \colon X \to Y $ is an admissible homomorphism. 
Then for all admissible subgroups $Z \subset X$ and $T \subset Y$, the groups   
$\varphi(Z)$, $\varphi^{-1} (T)$, and $Z \times T$ are admissible subgroups of $A^n$, $A^m$, and $A^{m+n}$ respectively. 
The identity map $\Id  \colon A \to A$ is  an admissible homomorphism and more generally, one has for every $n \geq 2$ that \[
\Delta^{(n)} = \{ (a, \dots, a) \in A^n\colon a \in A \} \in \HH_n. 
\]
\par 
Suppose that $\psi \colon Y \to Z$ is an admissible homomorphism where  $Z$ is an $\mathcal{H}$-admissible subgroup, then  $\psi \circ \varphi \colon X \to Z$ is also an admissible homomorphism. Note also that for $p \geq q \geq 0$, all the canonical projections $A^p \to A^q$ are  admissible homomorphisms.

\begin{example} 
\label{r:ad-ca} 
With respect to the canonical admissible Artinian 
structures of algebraic groups, resp. of compact Lie groups, resp. of Artinian groups, and of $R$-modules respectively (see Example~\ref{ex:canonical-admissible}), one find that all  homomorphisms of algebraic groups, resp. of compact Lie groups, resp. of Artinian groups, and morphisms of $R$-modules are admissible homomorphisms. 
\end{example}
 
\par 
The following auxiliary result says that the fibered products of admissible homomorphisms are also admissible homomorphisms. 
  \begin{lemma}
  \label{l:ad-mor}
Let $A$ be an admissible Artinian group structure. 
Let $m , n \geq 1$  and let $E$ be a finite set. 
Suppose that $\varphi_\alpha \colon A^m \to A^n$ is an admissible homomorphism for every $\alpha \in E$. 
Then the fibered product morphism  $\varphi_E \coloneqq (\varphi_\alpha)_{\alpha \in E} \colon A^m  \to (A^n)^E$, 
$\varphi_E(x) \coloneqq (\varphi_\alpha(x))_{\alpha \in E}$ for  all $x \in A^m$, 
is also an admissible homomorphism.      
\end{lemma} 

\begin{proof}
See \cite[Lemma~5.7]{phung-dcds}. 
\end{proof}

 \subsection{Admissible group subshifts}

We recall the natural notion of admissible group subshifts 
introduced in \cite{phung-2020}. However, we do not require the closedness property for subshifts  in this paper. 

\begin{definition}
\label{d:admissible-grp-subshift}
Let $G$ be a monoid and let $A$ be an admissible Artinian group structure. 
A subshift $\Sigma \subset A^G$ is called an \emph{admissible group subshift} 
if $\Sigma_E$ is an admissible subgroup of $A^E$ 
for every finite subset $E\subset G$.    
\end{definition}
\par 
The following example gives a natural class of admissible group subshifts. 
\par 
\begin{example} 
\label{ex:canonical-admissible-art-module} 
Let $G$ be a monoid and let $A$ be an Artinian module over a ring  $R$. Note that $A^G$ is an $R$-module with componentwise operations. Then every subshift 
$\Sigma \subset A^G$ which is also an $R$-submodule is automatically  an admissible group subshift of $A^G$ with respect to 
 the canonical admissible Artinian group  structure on $A$ (see Example~\ref{ex:canonical-admissible}).
 \end{example}

\subsection{Admissible group cellular automata} 

We extend the definition of admissible group cellular automata given  \cite[Definition~5.9]{phung-dcds} as follows. 

\begin{definition} 
Let $G$ be a monoid and let $A$ be an admissible Artinian group structure. 
Let $m, n \geq 1$ and let $\Sigma \subset (A^{m})^G$,   
$\Lambda \subset (A^{n})^G$ 
be admissible group subshifts. 
A map $\tau \colon \Sigma \to \Lambda$ is called an \emph{admissible group cellular automaton}  
if $\tau$ admits a finite memory set $M \subset G$ and an associated local defining map 
$\mu \colon \Sigma_M \to A^m$ which is an admissible homomorphism such that: 
\begin{equation*}
    \tau(x)(g) = \mu((g \star x)\vert_M), \quad \text{for all } x \in \Sigma, g \in G. 
\end{equation*}
\end{definition} 
\par 
Observe that we no longer  require admissible group cellular automata to extend to the full shift as in \cite[Definition~5.9]{phung-dcds}. We have the following key technical result: 
\par

\begin{lemma}
\label{l:tau-E-admissible}
Let $G$ be a monoid and let $A$ be an admissible Artinian group structure. 
Let $E \subset G$ be a finite subset and let $m, n \geq 1$. Let $\Sigma \subset (A^m)^G$ be an admissible group subshift. 
Let $\tau \colon \Sigma \to (A^n)^G$ be an admissible group cellular automaton 
with a given memory set $M \subset G$.  
Then the induced map  
$\tau_{E}^{+} \colon \Sigma_{M E} \to (A^n)^E$ defined by  
$\tau_E^{+}(c) \coloneqq \tau(x)\vert_E$ for all $c \in \Sigma_{ME}$ and  $x \in \Sigma$ such that $x \vert_{ME}=c$ 
is an admissible homomorphism. 
\end{lemma}
 
\begin{proof} 
Using Lemma~\ref{l:ad-mor}, the proof of the lemma is similar to  the proof of  \cite[Lemma~9.20]{phung-2020} in the case of group universes. 
\end{proof} 

\par 

The following theorem provides us with the methods to produce many  admissible group subshifts. 

\begin{theorem} 
\label{t:produce-admissible-grp-shift} 
Let $G$ be a countable monoid  
and let $A$ be an admissible Artinian group structure. 
Then the following hold for all $m, n \geq 1$: 
\begin{enumerate} [\rm (i)]
\item 
if $D \subset G$ is a finite subset and $P \subset A^D$ is an admissible subgroup,  
then $\Sigma(A^G; D, P)$ is an admissible group subshift of $A^G$. 
\item 
if $\tau \colon (A^m)^G \to (A^n)^G$ is an admissible group cellular automaton and 
  $\Sigma \subset (A^m)^G$, $\Lambda \subset (A^n)^G$  
are admissible group subshifts, then   
$\tau(\Sigma)$, $\tau^{-1}(\Lambda)$ are  respectively 
admissible group subshifts of $(A^n)^G$, $(A^m)^G$. 
\end{enumerate} 
\end{theorem}

\begin{proof}
See \cite[Theorem~5.11]{phung-dcds}. 
\end{proof}

\section{Left inverses of injective admissible group cellular automata} 
\label{s:techno-left-inverse}

We begin with the following auxiliary technical result on the left reversibility of injective admissible cellular automata on closed admissible group subshifts. 

\begin{lemma}
\label{l:techincal-left-inverse}
Let $G$ be a countable monoid. 
Let $A$ be an admissible Artinian group structure and let $\Sigma\subset A^G$ be a closed admissible group subshift. 
Let $\tau \colon \Sigma \to A^G$ be an injective admissible group  cellular automaton. 
Then there exists a finite subset $N\subset G$ such that 
\begin{enumerate}
\item[(P)]  
 \textit{for every} $x\in \tau(A^G)$, the element  
 $\tau^{-1}(x)(1_G)\in A$ \textit{depends uniquely on the restriction}  $x \vert_N$.
\end{enumerate}
\end{lemma}

\begin{proof}
Let us choose a finite memory set $M \subset G$ 
of $\tau$ such that $1_G \in M$. Let $\mu\colon \Sigma_M \to A$ be the corresponding local defining map of $\tau$. 
\par 
Since $G$ is countable, it admits an increasing sequence of finite subsets 
\[
M=E_0 \subset E_1 \subset \dots \subset E_n \subset  \cdots
\]
such that $G=\cup_{n\in \N} E_n$, i.e., $(E_n)_{n \in \N}$ forms an exhaustion of the monoid $G$. 
\par 
Let $\Gamma= \tau(A^G)$. Note that  since $\tau \colon A^G \to A^G$ is an injective group homomorphism, 
$\tau^{-1} \colon \Gamma \to A^G$ is clearly a group homomorphism. 
\par
We proceed by assuming on the contrary that there does not exist a finite subset $N\subset G$ verifying the property $(P)$. 
Consequently, there exist for every $n \geq 0$ two  configurations $x_n, y_n \in \Gamma$ such that we have 
\begin{equation}
\label{e:x-n-y-n-tau}
x_n\vert_{E_n}=y_n\vert_{E_n} \quad \text{and} \quad \tau^{-1}(x_n)(1_G)\neq \tau^{-1}(y_n)(1_G). 
\end{equation}
\par 
We denote $z_n\coloneqq x_ny_n^{-1} \in A^G$ for every $n \geq 0$. Let $e \in A$ be the neutral element. Then it follows that 
$z_n\vert_{E_n}=e^{E_n}$. Since $\tau^{-1}$ is a group homomorphism, we infer from \eqref{e:x-n-y-n-tau} that  $\tau^{-1}(z_n)(1_G)\neq e$.
\par 
Therefore, for 
$w_n\coloneqq \tau^{-1}(z_n)\vert_{E_nM}\in A^{E_nM}$, we find that: 
\begin{equation}
\label{e:tau-e-n-1}
\tau_{E_n}^+(w_n)=z_n\vert_{E_n}= e^{E_n} \quad \text{and } \quad   w_n(1_G) \neq e 
\end{equation}
where $\tau_{E_n}^+ \colon \Sigma_{E_n M } \to A^{E_n}$ is the admissible group homomorphism induced by the local defining map $\mu$ (see Lemma~\ref{l:tau-E-admissible}). 
\par 
For every $n\in \N$, we have an admissible subgroup of $\Sigma_{E_nM}$ defined by 
\[
U_n\coloneqq \Ker(\tau_{E_n}^+) \subset \Sigma_{E_nM} \subset A^{E_n M}.
\]
\par 
Note that $e^{E_n M}, w_n \in U_n$ for each $n\in \N$.  
Consider the canonical projection 
$\pi_{m,n} \colon A^{E_m M} \to A^{E_n M}$. Then 
it is clear that 
$\pi_{k,n}(U_{k}) \subset \pi_{m,n}(U_m)$ for all $k \geq m \geq n \geq 0$ since: 
\[
\pi_{k,n}(U_{k}) = \pi_{m,n}(\pi_{k,m}(U_{k})) 
\subset \pi_{m,n}(U_m).  
\]
\par 
Therefore, for every fixed $n\in \N$, we have a  decreasing sequence 
of admissible subgroups $(\pi_{m,n}(U_m))_{m \geq n}$ of $A^{E_n M}$.
\par 
Since $A$ is an admissible Artinian group structure,  $(\pi_{nm}(U_m))_{m \geq n}$ must stabilize. It follows that we can choose the smallest $r(n) \in \N$  depending on $n$ such that $r(n) \geq n$ and 
$\pi_{m,n}(U_m)= \pi_{r(n),n}(U_{r(n)})$ for all $m \geq r(n)$. 
\par 
For every $n \in \N$, let us denote 
\begin{equation}
\label{e:w-n-u-n-sigma} 
W_n \coloneqq  \pi_{r(n),n}(U_{r(n)}) \subset U_n \subset \Sigma_{E_nM}.
\end{equation} 
\par
Observe from our constructions that  for all $m \geq n \geq 0$,  the projection $\pi_{m,n} \colon A^{E_m M} \to A^{E_n M}$ induces by restriction a well-defined group homomorphism  
$p_{m,n} \colon W_m \to W_n$. Indeed, if $x \in W_m$ then note that $x\in  \pi_{k,m}(U_k)$ for $k = \max (r(m), r(n))$. Hence, we find that 
\[
\pi_{m,n}(x) \in \pi_{m,n}(\pi_{k,m}(U_k)) = \pi_{k,n}(U_k)= W_n. 
\]
\par 
\underline{\textbf{Claim}}:  $p_{m,n} \colon W_m  \to W_n$ is surjective for all $m\geq n \geq 0$. 
\par 
Indeed, let us fix $m \geq n \geq 0$ and  $x\in W_n$.  Since $W_n=\tau_{k,n}(U_k)$ for $k= \max (r(m), r(n))$, there exists $y\in U_k$ such that $p_{k,n}(y)=x$. If  we define $z = p_{k,m}(y) \in W_m$ then we find that 
\[
p_{m,n}(z) = p_{m,n}(p_{k,m}(y)) = p_{k,n}(y)= x.
\]
\par 
Hence, $W_n \subset p_{m,n}(W_n)$ and the claim is proved. 
\par 
We construct a sequence $(u_n)_{n \in \N}$ with   
$u_n \in W_n$ and $p_{n+1,n}(u_{n+1})=u_n$ for all $n\in \N$  as follows. 
First, we define $u_0 \coloneqq \pi_{r(0),0}(w_{r(0)}) \in W_0$. We infer from  \eqref{e:tau-e-n-1} that $u_0(1_G) \neq e$. 
 \par 
 Suppose that we have constructed $u_n \in W_n$  for some $n\in \N$. 
Since $p_{n+1,n}(W_{n+1})=W_n$, we can find and fix $u_{n+1} \in W_{n+1}$ such that 
\[
\pi_{n+1,n}(u_{n+1})=p_{n+1,n}(u_{n+1})=u_n.
\]
\par
Hence, we obtain by induction the sequence $(u_n)_{n \in \N}$ with   
$u_n \in W_n$ and $u_{n+1}\vert_{E_nM}=u_n$ for all $n\in \N$. Therefore, we 
can define $u \in A^G$ by setting $u\vert_{E_nM}= u_n \in W_n$ for every $n \in \N$. 
\par 
By construction, note that $u(1_G)= u_0(1_G) \neq 0$. Moreover, since we have $W_n \subset \Sigma_{M E_n}$ (see \eqref{e:w-n-u-n-sigma}) and $G= \cup_{n \in \N} E_n = \cup_{n \in \N}M E_n$ as $1_G \in M$, we deduce that $u$ belongs to the closure of $\Sigma$ in $A^G$ with respect to the prodiscrete topology. As $\Sigma$ is closed in $A^G$, it follows that $u \in \Sigma$. 
\par 
On the other hand, we infer from the relation $W_n \subset U_n= \Ker (\tau_{E_n}^+)$ that:
\[
\tau(u) \vert_{E_n} = \tau_{E_n}^+(u\vert_{M E_n})= \tau_{E_n}^+ (u_n) = e^{E_n}. 
\]
\par 
Therefore, $\tau(u)=e^G$ as $G=  \cup_{n \in \N}M E_n$. But $\tau(e^G)=e^g$ while $u \neq e^G$ since $u(1_G) \neq e$, we obtain a contradiction to the injectivity of $\tau$. This proves the existence of a finite subset $N \subset G$ satisfying $(P)$. The proof is thus complete. 
\end{proof}
\par 
We can now prove the following main result of the section. 

\begin{theorem}
\label{l:techno-left-inverse}
Let $G$ be a countable monoid. 
Let $A$ be an admissible Artinian group structure and let $\Sigma\subset A^G$ be a closed admissible group subshift. 
Let $\tau \colon \Sigma \to A^G$ be an injective admissible group cellular automaton. 
Let $\Gamma= \tau(\Sigma)$ then there exists a finite subset $N\subset G$ such that: 
\begin{enumerate}
\item[(K)]
for every finite subset $E \subset G$ containing $N$, 
there exists a group homomorphism  
$\eta_E \colon  \Gamma_E \to A$ 
such that for every $x\in \Sigma$:  
\begin{equation}
\label{e:left-inverse-local}
x(1_G)=\eta_E(\tau(x)\vert_E). 
\end{equation}
\end{enumerate}
\end{theorem}

\begin{proof}
We infer from Lemma~\ref{l:techincal-left-inverse} that there exists a finite subset $N \subset G$ such that for every $y \in \Gamma$, the element  $\tau^{-1}(y)(1_G)\in A$ depends only on the restriction $y \vert_N$.
Consequently, we have the following  well-defined map for every finite subset $E \subset G$ such that $N \subset E$: 
\begin{equation}
    \label{e:left-inverse-thm-1} 
    \eta_E \colon \Gamma_E \to A,\quad 
    y  \mapsto \tau^{-1}(z)(1_G),
\end{equation}
where $z \in \Gamma$ is any configuration extending $y \in \Gamma_E$. Now for $x \in \Sigma$, let $y= \tau(x)\vert_E$ then we deduce from \eqref{e:left-inverse-thm-1} that: 
\begin{equation*}
    \eta_E(\tau(x)\vert_E)= \eta_E(y)=  \tau^{-1}(\tau(x))(1_G)=x(1_G).
\end{equation*} 
\par 
Hence, $\eta_E$ satisfies the relation \eqref{e:left-inverse-local}. 
Observe that $\eta$ is clearly a group homomorphism since $\tau$ is an injective group homomorphism. The proof is thus complete. 
\end{proof} 

\section{Images of injective admissible group cellular automata}
\label{s:main-2-intro-admissible-general} 

The first goal of the present section is to give a proof of the following result which says that the image of every subshift of finite type under an injective admissible group cellular automata must be a subshift of finite type. 

\begin{theorem}
\label{t:main-2-intro-admissible}
Let $G$ be a countable monoid. 
Let $A$ be an admissible Artinian group structure and let $\Sigma\subset A^G$ be a closed admissible group subshift.  
Let $\tau \colon \Sigma \to A^G$ be an injective admissible group cellular automaton. Suppose that $\Delta \subset A^G$ is a subshift of finite type such that $\Delta \subset \Sigma$. Then $\tau(\Delta)$ is a subshift of finite type. 
\end{theorem}

\begin{proof}
Let us denote $\Gamma= \tau(\Sigma)$ and $X= \tau (\Delta)$. Let $\mu \colon \Sigma_M \to A$ be a local defining map of $\tau$ associated with a finite memory set $M\subset G$ such that $1_G \in M$. 
Let $D\subset G$ be a defining window of $\Delta$  so that 
$\Delta= \Sigma(A^G; D, \Delta_D)$. 
\par 
By Theorem~\ref{l:techno-left-inverse}, we can find a finite subset $N \subset G$ such that 
for every finite subset $E \subset G$ containing $N$, 
we have a group homomorphism  
$\eta_E \colon  \Gamma_E \to A$
such that for every $x\in \Sigma$:  
\begin{equation}
\label{e:left-inverse-local-1-2}
x(1_G)=\eta_E(\tau(x)\vert_E).
\end{equation} 
\par 
Therefore, together with  Lemma~\ref{l:tau-E-admissible}, we can replace without loss of generality $M$ and $N$ by $(M \cup N)D$. In particular, $D \subset M=N$ since $1_G \in M$ so that we have 
$\Delta = \Sigma(A^G; M, \Delta_M)$ by Lemma~\ref{l:window-change-sft}.  
\par 
Let us denote $\Lambda \coloneqq \Sigma(A^G; M^2, \Gamma_{M^2})$ and $Y= \Sigma(A^G; M^2, X_{M^2})$.  
Consider the cellular automaton $\sigma \colon \Lambda \to A^G$ admitting $\eta_M\colon \Gamma_M \to A$ as a local defining map (note that $\Gamma_M=\Lambda_M$ as $M \subset M^2$).  . We deduce from the relation  \eqref{e:left-inverse-local-1-2} and the $G$-equivariance of $\tau$ and $\sigma$ that for every $x \in \Sigma$ and every $g \in G$, we have: 
\begin{align} 
\label{e:main-ad-grp-proof-1}
    \sigma(\tau(x))(g)
    & = \eta_M((g \star \tau(x))\vert_M) \nonumber \\
    & = \eta_M(\tau(g \star x)\vert_M)\nonumber \\
    & = (g \star x)(1_G) \nonumber \\
    &= x(g). \numberthis 
\end{align} 
\par 
Consequently, 
$\sigma \circ \tau  \colon \Sigma \to \Sigma$ is the identity map of $\Sigma$. In particular, the restriction $\sigma \circ \tau\vert_\Delta$ is the identity map of $\Delta$. 
We can thus conclude from Theorem~\ref{t:criterion-finite-type}    that 
\[
X = Y= \Sigma(A^G; M^2, X_{M^2})
\]
is a subshift of finite type of $A^G$. The proof is thus complete. 
\end{proof}

\par 
As an immediate application of Theorem~\ref{t:main-2-intro-admissible}, we obtain: 

\begin{corollary}
\label{c:image-sft-again-sft}
Let $G$ be a countable monoid and let $A$ be an admissible Artinian group structure. Let  $\Sigma\subset A^G$ be an admissible group subshift of finite type. Then for every  
 injective admissible group cellular automaton $\tau \colon \Sigma \to A^G$, the image  
$\tau(\Sigma)$ is an admissible group subshift of finite type.
\end{corollary}

\begin{proof}
It follows from Theorem~\ref{t:main-2-intro-admissible} that $\tau(\Sigma)$ is a subshift of finite type. Since $\Sigma$ is an admissible group subshift, we infer from Theorem~\ref{t:produce-admissible-grp-shift} that $\tau(\Sigma)$ is an admissible group subshift of $A^G$. The conclusion thus follows. 
\end{proof}

It turns out that the converse of Theorem~\ref{t:main-2-intro-admissible} also holds as follows. The proof is similar to that of Theorem~\ref{t:converse-sft-image}. 

\begin{theorem}
\label{t:main-2-intro-admissible-converse} 
Let $G$ be a countable monoid. 
Let $A$ be an admissible Artinian group structure and let $\Sigma\subset A^G$ be a closed admissible group subshift.  
Let $\tau \colon \Sigma \to A^G$ be an injective admissible group cellular automaton. Suppose that $\Delta \subset A^G$ is a subshift  such that $\Delta \subset \Sigma$ and $\tau(\Delta)$ is a subshift of finite type. Then $\Delta$ is also a subshift of finite type.  
\end{theorem}

\begin{proof}
We will proceed with several similar constructions and notation as in Theorem~\ref{t:main-2-intro-admissible}. 
Hence, we denote $\Gamma= \tau(\Sigma)$ and $X= \tau (\Delta)$ and note that $X \subset \Gamma$ since $\Delta \subset \Sigma$ by hypothesis. 
\par 
Since $X$ is a subshift of finite type by hypothesis, we can choose a finite defining window $D\subset G$ of $X$ so that  
$X= \Sigma(A^G; D, X_D)$. 
Let $\mu_M \colon \Sigma_M \to A$ be a local defining map of $\tau$ associated with a finite memory set $M\subset G$ such that $1_G \in M$. 
\par 
By Theorem~\ref{l:techno-left-inverse}, we can find a finite subset $N \subset G$ such that 
for every finite subset $E \subset G$ containing $N$, 
we have a group homomorphism  
$\eta_E \colon  \Gamma_E \to A$
such that for every $x\in \Sigma$, we have  
$x(1_G)=\eta_E(\tau(x)\vert_E)$. 
\par 
By replacing $M$ and $N$ by $(M \cup N)D$,  we can suppose without loss of generality that $D \subset M=N$. Hence, by  Lemma~\ref{l:window-change-sft}, we can write: 
\begin{equation}
\label{e:main-ad-group-sigma-fsft-converse-1}
      X= \Sigma(A^G; M, X_M). 
\end{equation}
\par 
Let us define $Z\coloneqq \Sigma(A^G; M^2, \Delta_{M^2}) \subset A^G$ then  $\Delta \subset Z$. Moreover, we infer from Lemma~\ref{l:restriction-subshift-equality} that $Z_M = \Sigma_M$ since $M \subset M^2$ as $1_G \in M$. 
\par 
Therefore, the local defining map $\mu_M$ determines a cellular automaton 
$\pi \colon Z \to A^G$ whose restriction to $\Delta$ coincides   with $\tau$, that is, $\pi\vert_\Delta= \tau$.   
\par 
Denote $\Lambda \coloneqq \Sigma(A^G; M^2, \Gamma_{M^2}) \subset A^G$ and consider the cellular automaton $\sigma \colon \Lambda \to A^G$ admitting $\eta_M\colon \Gamma_M \to A$ as a local defining map (note that $\Gamma_M=\Lambda_M$ by Lemma~\ref{l:restriction-subshift-equality} as $M \subset M^2$).  
\par 
Since 
$\sigma \circ \tau$ is the identity map on $\Sigma$ as we have seen in the proof of Theorem~\ref{t:main-2-intro-admissible}, we have $\sigma(X)=\sigma(\tau(\Delta))= \Delta$. On the other hand, since $\tau(\Delta)= X$ and $\Delta \subset \Sigma$, 
we deduce immediately that the restriction 
$\pi \circ \sigma \vert_X= \tau \circ \sigma \vert_X$ acts as the identity map on $X$.    
\par 
Hence, it follows from Theorem~\ref{t:criterion-finite-type} applied to  the subshift of finite type $X$ and the cellular automata $\sigma\vert_X \colon X \to A^G$, $\tau\vert_\Delta \colon \Delta \to A^G$ that 
\[ 
\Delta = Z= \Sigma(A^G; M^2, \Delta_{M^2})
\] 
so $\Delta$ is a subshift of finite type. The proof is thus complete. 
\end{proof}
\par
In parallel to Theorem~\ref{t:converse-sft-image},  we can now establish the following general result from which we deduce easily  Theorem~\ref{t:main-corollary-admissible-intro} in the Introduction.  
\begin{theorem}
\label{t:main-2-intro-admissible-general} 
Let $G$ be a countable monoid. 
Let $A$ be an admissible Artinian group structure and let $\Sigma\subset A^G$ be a closed admissible group subshift. Let $\Delta \subset A^G$ be a subshift such that $\Delta \subset \Sigma$. Suppose that $\tau \colon \Sigma \to A^G$ is an injective admissible group cellular automaton. Then the following hold: 
\begin{enumerate} [\rm (i)] 
    \item $\Delta$ is a subshift of finite type if and only if  so is  $\tau(\Delta)$; 
    \item 
    $\Delta$ is a sofic subshift if and only if so is  $\tau(\Delta)$.   
\end{enumerate}
\end{theorem}

\begin{proof}
The point (i) results directly from the combination of Theorem~\ref{t:main-2-intro-admissible} and  Theorem~\ref{t:main-2-intro-admissible-converse}. For (ii), suppose first that $\Delta$ is a sofic  subshift. Then $\Delta$ is the image of a subshift of finite type $X$ under a cellular automaton $\pi$. Since the composition of two cellular automata is also a cellular automaton, it follows that $\tau(\Delta)= \tau(\pi(X))$ is also a sofic subshift. 
\par 
Conversely, suppose that $ \tau(\Delta) \subset A^G$ is a sofic subshift. Hence, $\tau(\Delta)$ is the image of a subshift of finite type $Y$ under a cellular automaton $\pi$.  
Let $\Gamma= \tau(\Sigma) \subset A^G$ then as in the proof of Theorem~\ref{t:main-2-intro-admissible}, there exists a cellular automaton $\sigma \colon \Gamma \to A^G$ such that $\sigma (\Gamma)=\Sigma$ and that $\sigma \circ \tau$ is the identity map of $\Sigma$. Since $\Delta \subset \Sigma$, it is clear that \[
\Delta = \sigma (\tau(\Delta)) =
\sigma (\pi(Y)) 
\] 
and we can again conclude that $\Delta$ is a sofic subshift. 
The proof is thus complete. 
\end{proof}

\section{Proof of Theorem ~\ref{t:main-corollary-admissible-intro}}
\label{s:proof-thm-b}
We can now deduce  Theorem~\ref{t:main-corollary-admissible-intro} from Theorem~\ref{t:main-2-intro-admissible-general} using a general reduction step to the case of one alphabet as follows. 

\begin{proof}[Proof of Theorem~\ref{t:main-corollary-admissible-intro}]  
Let $M\subset G$ be a finite memory set of $\tau$ such that $1_G \in M$ and let $\mu \colon \Sigma_M \to B$ be the corresponding local defining map which is also a module homomorphism. 
\par 
Let $\Delta \subset A^G$ be a subshift contained in $\Sigma$. Since $A$ and $B$ are Artinian modules over a ring that we denote by $R$, the direct sum $S=A\oplus B$ is also an Artinian $R$-module. 
\par 
For every subset $E \subset G$, we denote by $\pi_E \colon S^E \to A^E$ the canonical projection. Let us consider the following subshifts $\Delta(S) \subset \Sigma(S)$ of $S^G$: 
\[
\Delta(S)\coloneqq \pi_G^{-1}(\Delta), \quad   \Sigma(S) \coloneqq \pi_G^{-1}(\Sigma).
\]
\par 
It is clear that $\Sigma(S)$ is also a closed subshift submodule of $S^G$ since $\Sigma$ is a closed subshift submodule of $A^G$. 
Moreover, we can verify without difficulty  that $\Delta$ is a  subshift of finite type, resp. a sofic subshift, if and only if so is the subshift $\Delta(S)$. Note also that $\Sigma(S)_M = \pi_M^{-1}(\Sigma_M)$. 
\par 
We now define an $R$-module morphism $\mu_S \colon \Sigma(S)_M \to S$ as follows. Let $s \in \Sigma(S)_M \subset S^M$, we define $x \in A^M$ and $y \in B^M$ by the direct sum  decomposition $s(g)=(x(g),y(g)) \in A \oplus B$ for all $g \in M$. Then we simply set $\mu_S(s)\coloneqq (\mu(x), y(1_G))$. Using this formula, it is not hard to check that $\mu_S$ is a morphism of $R$-modules. 
\par 
We denote by $\tau_S \colon \Sigma(S) \to S^G$ the cellular automaton which admits $\mu_S$ as a local defining map. Then $\tau$ is a  homomorphism of $R$-modules and we deduce immediately from the construction that for $s=(x,y) \in \Sigma(S)$ where $x \in \Sigma$ and $y \in B^G$, we have the following relation:   \begin{equation}
\label{e:tau-s-tau-reduction}
    \tau_S(s)=(\tau(x), y)
\end{equation} 
\par 
Consequently,   $\tau_S(\Delta(S))=  \pi_G^{-1}(\tau(\Delta))$ and it follows that $\tau(\Delta)$ is a subshift of finite type, resp. a sofic subshift, if and only if so is $\tau_S(\Delta(S))$. 
\par 
Since $\tau$ is injective, 
we infer from  \eqref{e:tau-s-tau-reduction} that $\tau_S$ is also injective. Therefore, with respect to the canonical admissible Artinian group structure of $S$ as an Artinian module, Theorem~\ref{t:main-2-intro-admissible-general} implies that the subshift $\tau_S(\Delta(S))\subset S^G$ is of finite type, resp. sofic, if and only if so is the subshift $\Delta(S)\subset S^G$. \par 
Hence, the above discussions show that $\tau(\Delta)\subset A^G$ is a subshift of finite type, resp. sofic, if and only if so is $\Delta$. The proof is thus complete. 
\end{proof}

\section{Application on injective morphisms of symbolic group varieties} 
\label{s:app-sym-group-var}
\par 
Given a monoid $G$ and  algebraic groups $X, Y$ over an algebraically closed field $k$ (cf.~\cite{milne}). Let $A=X(k)$ and $B=Y(k)$ then following \cite{phung-israel}, a subshift $\Sigma \subset A^G$ is called a closed \emph{algebraic group subshift} if it is closed and for every finite subset $E \subset G$, the restriction $\Sigma_E$ is an algebraic subgroup of $A^E$. Given a closed algebraic group subshift $\Sigma \subset A^G$, a cellular automaton $\tau \colon \Sigma \to B^G$ is called an \emph{algebraic group cellular automaton} if it admits a local defining map $\mu\colon \Sigma_M \to B$ for some finite memory $M \subset G$ such that $\mu$ is a $k$-homomorphism of algebraic groups. 
\par 
As an another direct  application of Theorem~\ref{t:main-2-intro-admissible-general}, we obtain the following extension of Theorem~\ref{t:sft-image-intro-main-alg} in the case of algebraic group alphabets over any algebraically closed field (not necessarily uncountable). 
\par 

\begin{theorem}
\label{t:sym-grp-var-inj}
Let $G$ be a countable monoid and let $X, Y$ be algebraic groups over an algebraically closed field $k$. Let $\Sigma \subset X(k)^G$ be a closed algebraic group subshift. Suppose that $\tau \colon \Sigma \to Y(k)^G$ is an injective algebraic group cellular automaton. Then for every subshift $\Delta \subset X(k)^G$ such that $\Delta \subset \Sigma$, the subshift  $\tau(\Delta) \subset Y(k)^G$ is of finite type, resp. a sofic subshift, if and only if so is $\Delta$.     
\end{theorem}

\begin{proof}
It suffices to apply  Theorem~\ref{t:main-2-intro-admissible-general} after a reduction procedure described in the proof of Theorem~\ref{t:main-corollary-admissible-intro} to reduce to the case when $X=Y$. We only remark here that instead of taking the direct sum $S$ of the two module alphabets as in the proof of Theorem~\ref{t:main-corollary-admissible-intro}, we simply consider the fibered product $S= X \times_k Y$ which is an algebraic group over $k$. 
\par 
Note that by Remark~\ref{r:ad-ca}, we find that $\tau$ is an admissible group cellular automata with respect to the admissible Artinian group structures associated with algebraic groups described in Example~\ref{ex:canonical-admissible}. Likewise, $\Sigma$ is a closed admissible group subshift of $X(k)^G$. 
The proof is complete. 
\end{proof}

\bibliographystyle{siam}

\end{document}